\documentclass{article}
\usepackage{amssymb}
\usepackage[margin=2.5cm]{geometry}
\usepackage{amsmath}
\usepackage{amsthm}
\usepackage{mathrsfs}
\usepackage{subcaption}
\usepackage{tikz}
\usepackage{listings}
\usepackage{float}
\usepackage{marvosym}

\usepackage{fancyhdr}
\newtheorem{defn}{Definition}[section]

\newtheorem{prop}[defn]{Proposition}

\newtheorem{theorem}[defn]{Theorem}

\usetikzlibrary{positioning}
\usetikzlibrary{matrix,arrows}
\usetikzlibrary{shapes.geometric}
\usetikzlibrary{decorations.pathreplacing}
\usetikzlibrary{arrows.meta}

\newcommand{\E}{\mathbb{E}}
\newcommand{\Prob}{\mathbb{P}}

\makeatletter
\DeclareRobustCommand{\pder}[1]{%
  \@ifnextchar\bgroup{\@pder{#1}}{\@pder{}{#1}}}
\newcommand{\@pder}[2]{\frac{\partial#1}{\partial#2}}
\makeatother

\begin{document}

\title{The nonhomogeneous frog model on $\mathbb{Z}$}

\author{Josh Rosenberg}
\date{}
\maketitle

\begin{abstract}We examine a system of interacting random walks with leftward drift on $\mathbb{Z}$, which begins with a single active particle at the origin and some distribution of inactive particles on the positive integers.  Inactive particles become activated when landed on by other particles, and all particles beginning at the same point posses equal leftward drift.  Once activated, the trajectories of distinct particles are independent.  This system belongs to a broader class of problems involving interacting random walks on rooted graphs, referred to collectively as the frog model.  Additional conditions that we impose on our model include that the number of frogs (i.e. particles) at positive integer points is a sequence of independent random variables which is increasing in terms of the standard stochastic order, and that the sequence of leftward drifts associated with frogs originating at these points is decreasing.  Our results include sharp conditions with respect to the sequence of random variables and the sequence of drifts, that determine whether the model is transient (meaning the probability infinitely many frogs return to the origin is $0$) or non-transient.  Several, more specific, versions of the model described will also be considered, and a cleaner, more simplified set of sharp conditions will be established for each case.
\end{abstract}

\section{Introduction}
The frog model is a system of interacting random walks on a rooted graph.  It begins with a single ``active" frog at the root and some distribution of sleeping frogs (either random or deterministic) at the non-root vertices.  The active frog performs a discrete-time nearest neighbor random walk on the graph (biased or unbiased) and any time an active frog lands on a vertex containing sleeping frogs, all of these frogs become active and begin performing their own discrete-time nearest neighbor random walks, activating sleeping frogs along the way.  Previous work on the frog model has included looking at the model on infinite n-ary trees as well as on Euclidean lattices.  In particular, a number of people have studied a variety of different versions of the frog model on $\mathbb{Z}$, often focusing on establishing conditions that determine whether the model is recurrent or transient with respect to the number of distinct frogs that visit the root.  This paper will focus on exploring several of these models while building on, expanding, and synthesizing some of the existing results pertaining to them.

There are three existing results, each addressing a different version of the frog model on $\mathbb{Z}$, that serve as a jumping off point for the present work.  The first concerns a model in which all nonzero vertices contain an i.i.d. number of sleeping frogs, and activated frogs perform mutually independent random walks that go left with probability $p$ (where $\frac{1}{2}<p<1$) and right with probability $1-p$.  In \cite{Gantert} Nina Gantert and Philipp Schmidt prove that if $\eta$ represents a random variable with the same distribution as the number of sleeping frogs at each nonzero vertex, then\begin{equation}\label{gantschmtresult}\Prob_{\eta}(\text{the origin is visited i.o.})=\left\{\begin{array}{ll}0 &\text{if }\E[\text{log}^+\eta]<\infty\\1 &\text{if }\E[\text{log}^+\eta]=\infty\end{array}\right.\end{equation}(note that this condition does not depend on the particular value of $p$).

The second result that served to motivate the present investigations involved a model in which negative integer vertices contain no sleeping frogs and positive integer vertices each contain a single sleeping frog.  For each $n>0$ the frog at $x=n$ (if activated) performs a random walk (independently of the other active frogs) that goes left with probability $p_n$ (with $\frac{1}{2}<p_n<1$), and right with probability $1-p_n$ (i.e. the particular drift value depends on where the frog originated).  In \cite{BMZ} Daniela Bertacchi, Fabio Prates Machado, and Fabio Zucca established (in addition to a number of other results) that if there exists some increasing sequence of positive integers $\left\{n_k\right\}_{k\in\mathbb{N}}$ such that\begin{equation}\label{bmzrelresult}\sum_{k=0}^{\infty}\prod_{i=0}^{n_k}\Big(1-\Big(\frac{1-p_i}{p_i}\Big)^{n_{k+1}-i}\Big)<\infty\end{equation}then the model is non-transient (i.e. infinitely many frogs hit the origin with positive probability).

The third and final result which this paper builds on again looks at a frog model on $\mathbb{Z}$ for which no sleeping frogs reside to the left of the origin, and activated frogs perform random walks with leftward drift.  This time the number of sleeping frogs $X_j$ at $x=j$ (for $j\geq 1$) has distribution $\text{Poiss}(\eta_j)$, where the $X_j$'s are mutually independent and $\left\{\eta_j\right\}$ is an increasing sequence.  At each step activated frogs go left with probability $p$ (for $\frac{1}{2}<p<1$) and right with probability $1-p$.  This model was presented in \cite{Rosenberg} by the author, where it was established that the model is non-transient if and only if\begin{equation}\label{myres}\sum_{j=1}^{\infty}e^{-\frac{1-p}{2p-1}\eta_j}<\infty\end{equation}

\bigskip 
\noindent{\bf Statement and discussion of results.} The first result that we'll present establishes a sharp condition distinguishing between transience and non-transience for a more general frog model on $\mathbb{Z}$ that subsumes all three of the models described above.  In this model points to the left of the origin contain no sleeping frogs and, for $j\geq 1$, the number of sleeping frogs at $x=j$ is a random variable $X_j$, where the $X_j$'s are independent, non-zero with positive probability, and where $X_{j+1}\succeq X_j$ (here $``\succeq"$ represents stochastic dominance).  In addition, for each $j\geq 1$ frogs originating at $x=j$ (if activated) go left with probability $p_j$ (where $\frac{1}{2}<p_j<1$) and right with probability $1-p_j$, where the $p_j$'s are decreasing and the random walks are all mutually independent (the frog beginning at the origin goes left with probability $p_0$, where $p_0$ also satisfies $\frac{1}{2}<p_0<1$).  This model will be referred to as the nonhomogeneous frog model on $\mathbb{Z}$, and the sharp condition we eluded to will come in the form of the following theorem.

\begin{theorem}\label{theorem:neat1}
Let $f_j$ be the probability generating function of $X_j$ for the nonhomogeneous frog model on $\mathbb{Z}$.  The model is transient if and only if\begin{equation}\label{generalform}\sum_{n=2}^{\infty}\prod_{j=1}^{n-1}f_j\Big(1-\Big(\frac{1-p_j}{p_j}\Big)^{n-j}\Big)=\infty\end{equation}
\end{theorem}

\medskip
After establishing Theorem \ref{theorem:neat1} the focus will shift towards showing how it can be applied in a number of more specific cases.  The first application of the theorem will involve the Gantert and Schmidt model from \cite{Gantert}, and will entail showing how \eqref{gantschmtresult} can be achieved quite easily using the formula \eqref{generalform}.  After this, \ref{theorem:neat1} is used to obtain a formula (see Theorem \ref{theorem:neat3}) that provides a sharp condition distinguishing between transience and non-transience in the case where the $X_j$'s are i.i.d. and which, for the particular case where $X_j=1$, builds on the result from \cite{BMZ} by giving a sharp result that supersedes the soft condition in \eqref{bmzrelresult} and, for the case where $p_j=\frac{1}{2}+\frac{C}{\text{log}\ j}$ (for all but finitely many $j$), implies the existence of a phase transition at $C=\frac{\pi^2}{24}$.  Finally, \ref{theorem:neat1} will also be employed to obtain a formula that builds on the result from \cite{Rosenberg} by generalizing \eqref{myres} to cases where the $p_j$'s are not constant.  For these last two results, the proofs will require some light assumptions relating to the concavity of the sequences $\left\{p^{-1}_j\right\}$ and $\left\{\lambda_j\right\}$ (where $\lambda_j$ represents the Poisson mean of the distribution of $X_j$ in the final model discussed).

\section{Transience vs. non-transience for the general case}
  
  \subsection{$M_j\text{ and }N_j$} \label{ss:process N&M}
  
  In order to move towards a proof of Theorem \ref{theorem:neat1}, we begin by defining the process $\left\{M_j\right\}$ where, for each $j\geq 1$, $M_j$ represents the number of frogs originating in $\left\{0,1,\dots,j-1\right\}$ that ever hit the point $x=j$.  $\left\{M_j\right\}$ is now identified with a triple $(\Omega,\mathcal{F},{\bf P})$ defined as follows: $\Omega$ will represent the set of all functions $\omega:\mathbb{Z}^+\rightarrow\mathbb{N}$ (i.e. the set of all possible trajectories of $\left\{M_j\right\}$), $\mathcal{F}$ will represent the $\sigma$-field on $\Omega$ generated by the finite dimensional sets, and ${\bf P}$ will refer to the probability measure induced on $(\Omega,\mathcal{F})$ by the process $\left\{M_j\right\}$.  Since $\Prob(X_n\geq 1)\geq\Prob(X_1\geq 1)>0\ \forall\ n\geq 1$ (recall $X_{j+1}\succeq X_j\ \forall\ j\geq 1$) and the $X_j$'s are independent, it follows from Borel-Cantelli II that $\left\{X_j\geq 1\ \text{i.o.}\right\}\ \text{a.s.}$  Additionally, since each activated frog performs a random walk with nonzero leftward drift, this means that each activated frog will eventually hit the origin with probability 1.  Coupling this with the fact that $\left\{X_j\geq 1\ \text{i.o.}\right\}\ \text{a.s.}\implies\sum_{j=1}^{\infty}X_j=\infty\ \text{a.s.}$, along with the implication $M_l=0\implies M_j=0\ \forall\ j>l$, we find that\begin{equation}\label{impequiv}\left\{\text{infinitely many frogs hit the origin}\right\}\iff\text{min}\ M_j>0\end{equation}Now on account of \eqref{impequiv}, it follows that in order to establish Theorem \ref{theorem:neat1}, it suffices to show that\begin{equation}\label{fthsuff}\text{min}\ M_j=0\ \ {\bf P}-\text{a.s.}\iff\sum_{n=2}^{\infty}\prod_{j=1}^{n-1}f_j\Big(1-\Big(\frac{1-p_j}{p_j}\Big)^{n-j}\Big)=\infty\end{equation}With this in mind, we define a new model which we'll call the $\text{F}^+$ model.  This model will resemble the non-homogeneous frog model on $\mathbb{Z}$ in that the distribution of the number of frogs beginning at every vertex will be the same in the two cases, as will the drifts of the active frogs.  The only difference will be that in the $\text{F}^+$ model all frogs will begin as active frogs (i.e. they do not need to be landed on to be activated).  The next step is to now use the $\text{F}^+$ model to define the process $\left\{N_j\right\}$ where, for each $j\geq 1$, $N_j$ equals the number of frogs originating in $\left\{0,1,\dots,j-1\right\}$ that ever hit the point $x=j$ in the $\text{F}^+$ model (i.e. $\left\{N_j\right\}$ is identical to $\left\{M_j\right\}$ except that the $\text{F}^+$ model replaces the non-homogeneous frog model on $\mathbb{Z}$ in the definition).  $\left\{N_j\right\}$ can now be identified with the triple $(\Omega,\mathcal{F},{\bf Q})$, where ${\bf Q}$ will refer to the probability measure induced on $(\Omega,\mathcal{F})$ by the process $\left\{N_j\right\}$.  Having defined this construction, we'll now establish the following proposition, which will serve as the key step in proving Theorem \ref{theorem:neat1}.
  
  \medskip
  \begin{prop}\label{prop:neat2}
  Define the random variable $K(\omega)=\#\left\{j\in\mathbb{Z}^+:\omega(j)=0\right\}$.  Then ${\bf Q}(K=\infty)=1$ if and only if\begin{equation}\label{tightsum}\sum_{n=2}^{\infty}\prod_{j=1}^{n-1}f_j\Big(1-\Big(\frac{1-p_j}{p_j}\Big)^{n-j}\Big)=\infty\end{equation}If \eqref{tightsum} does not hold then ${\bf Q}(K=\infty)=0$.
  \end{prop}
  
  \medskip
  \noindent $Remark.$ It is worth noting that it cannot be assumed that $\left\{M_j\right\}$ and $\left\{N_j\right\}$ are Markov processes since $M_j$ ($N_j$ resp.) only gives the number of frogs originating to the left of the point $x=j$ that ever hit $x=j$, rather than also providing the information about where each such frog originated (a significant detail, since frog origin determines drift).  Nevertheless, because the only conditioning we will do with respect to these two processes will involve conditioning on $M_j$ ($N_j$ resp.) equalling $0$, they prove to be sufficient for our purposes.
  
  \begin{proof}[Proof of Proposition \ref{prop:neat2}]
  By a simple martingale argument the probability a frog starting at $x=j$ ever hits $x=n$ (for $n>j$) is $\Big(\frac{1-p_j}{p_j}\Big)^{n-j}$.  Hence, the probability that no frogs beginning at $x=j$ ever hit $x=n$ is$$\sum_{i=0}^{\infty}\Prob(X_j=i)\Big(1-\Big(\frac{1-p_j}{p_j}\Big)^{n-j}\Big)^i=f_j\Big(1-\Big(\frac{1-p_j}{p_j}\Big)^{n-j}\Big)$$It then follows that for every $n\geq 1$ we have$${\bf Q}(\omega(n)=0)=\Big(1-\Big(\frac{1-p_0}{p_0}\Big)^n\Big)\prod_{j=1}^{n-1}f_j\Big(1-\Big(\frac{1-p_j}{p_j}\Big)^{n-j}\Big)$$
  $$\implies{\bf E}[K]=\frac{2p_0-1}{p_0}+\sum_{n=2}^{\infty}\Big(1-\Big(\frac{1-p_0}{p_0}\Big)^n\Big)\prod_{j=1}^{n-1}f_j\Big(1-\Big(\frac{1-p_j}{p_j}\Big)^{n-j}\Big)$$(where ${\bf E}$ refers to expectation with respect to the probability measure ${\bf Q}$).  Since $\Big(1-\Big(\frac{1-p_0}{p_0}\Big)^n\Big)\rightarrow 1$ as $n\rightarrow\infty$, this means\begin{equation}\label{expequivsum}{\bf E}[K]<\infty\iff\sum_{n=2}^{\infty}\prod_{j=1}^{n-1}f_j\Big(1-\Big(\frac{1-p_j}{p_j}\Big)^{n-j}\Big)<\infty\end{equation}It now immediately follows that if the right side of \eqref{expequivsum} holds, then ${\bf Q}(K=\infty)=0$.  Hence, to prove the proposition it suffices to establish the implication ${\bf Q}(K=\infty)<1\implies{\bf E}[K]<\infty$ (note this is just the contrapositive of ${\bf E}[K]=\infty\implies{\bf Q}(K=\infty)=1$).
  
  Now since the event $\left\{K=\infty\right\}$ cannot depend on the behavior of the frogs from any finite collection of vertices (for the process $\left\{N_j\right\}$), it follows that ${\bf Q}(K=\infty|\omega(1)=0)={\bf Q}(K=\infty|\omega(1)=1)$, which in turn establishes the implication\begin{equation}\label{loneimp}{\bf Q}(K=\infty)<1\implies{\bf Q}(1\leq K<\infty)>0\end{equation}Next define $V_n=\left\{\omega\in\Omega:\omega(j)>0\ \forall\ j>n\right\}$ and assume ${\bf Q}(K=\infty)<1$.  Letting ${\bf Q}^{(n)}$ denote the probability measure obtained by conditioning on the event $\omega(n)=0$, \eqref{loneimp} then implies that there must exist $L\geq 1$ such that ${\bf Q}^{(L)}(V_L)>0$.  Additionally, because $X_{i_1+i_2}\succeq X_{i_1}\ \forall\ i_1,i_2\geq 1$ (since $X_{i+1}\succeq X_i\ \forall\ i\geq 1$ and $\succeq$ is transitive) and because the sequence of drifts $\left\{p_j\right\}$ is decreasing with respect to $j$, this implies that for any $L'>L$ the models $\big(\text{F}^+|N_L=0\big)$ and $\big(\text{F}^+|N_{L'}=0\big)$ (i.e. the $\text{F}^{+}$ model with all frogs to the left of the point $x=L$ removed) can be coupled so that (i) every frog originating at $x=L+j$ in $\big(\text{F}^+|N_L=0\big)$ corresponds to a particular frog originating at $x=L'+j$ in the coupled model $\big(\text{F}^+|N_{L'}=0\big)$, and (ii) whenever a frog in $\big(\text{F}^+|N_L=0\big)$ takes a step to the right, the corresponding frog in $\big(\text{F}^+|N_{L'}=0\big)$ does as well.  Letting $K_n(\omega)=\#\left\{j>n:\omega(j)=0\right\}$, the above coupling then implies that\begin{equation}\label{sdimpineq}\big(K_L|\omega(L)=0\big)\succeq\big(K_{L'}|\omega(L')=0\big)\implies{\bf Q}^{(L')}(V_{L'})\geq{\bf Q}^{(L)}(V_L)\end{equation}Now if we define the stopping times $T_n$ where $T_1(\omega)=\text{min}\left\{j\geq 1:\omega(L+j)=0\right\}$ and, for $n\geq 2$, $T_n(\omega)=\text{min}\left\{j>T_{n-1}(\omega):\omega(L+j)=0\right\}$, we find that for every $n\geq 2$\begin{equation}\label{eqandineq}{\bf Q}^{(L)}(K_L\geq n)=\sum_{j=1}^{\infty}{\bf Q}^{(L)}(T_{n-1}=j){\bf Q}^{(L+j)}(V^c_{L+j})\leq{\bf Q}^{(L)}(V^c_L){\bf Q}^{(L)}(K_L\geq n-1)\end{equation}(where the inequality follows from \eqref{sdimpineq}).  From this it then follows that for $n\geq 1$ $${\bf Q}^{(L)}(K_L\geq n)\leq\big(1-{\bf Q}^{(L)}(V_L)\big)^n\implies{\bf E}[K_L|\omega(L)=0]\leq\sum_{n=1}^{\infty}\big(1-{\bf Q}^{(L)}(V_L)\big)^n=\frac{1-{\bf Q}^{(L)}(V_L)}{{\bf Q}^{(L)}(V_L)}<\infty$$Since ${\bf E}[K_L]\leq {\bf E}[K_L|\omega(L)=0]$ and ${\bf E}[K]\leq L+{\bf E}[K_L]$, we find that$${\bf E}[K]\leq L+\frac{1-{\bf Q}^{(L)}(V_L)}{{\bf Q}^{(L)}(V_L)}<\infty$$Hence, we've established the implication ${\bf Q}(K=\infty)<1\implies{\bf E}[K]<\infty$, which then gives the implication ${\bf E}[K]=\infty\implies{\bf Q}(K=\infty)=1$, thus completing the proof of the proposition.
  \end{proof}
  
  \subsection{Proving Theorem 1.1} \label{ss:pr1.1}
  
  \begin{proof}[Proof of Theorem \ref{theorem:neat1}]
  Coupling the fact that Theorem \ref{theorem:neat1} is equivalent to \eqref{fthsuff} with the fact that ${\bf P}(\text{min}\ \omega(j)=0)=1\iff{\bf Q}(K\geq 1)=1$, we find the task of proving Theorem \ref{theorem:neat1} is reduced to establishing that\begin{equation}\label{importanteq}{\bf Q}(K\geq 1)=1\iff\sum_{n=2}^{\infty}\prod_{j=1}^{n-1}f_j\Big(1-\Big(\frac{1-p_j}{p_j}\Big)^{n-j}\Big)=\infty\end{equation}Noting that the implication\begin{equation}\label{sumpdimpeq}\sum_{n=2}^{\infty}\prod_{j=1}^{n-1}f_j\Big(1-\Big(\frac{1-p_j}{p_j}\Big)^{n-j}\Big)=\infty\implies{\bf Q}(K\geq 1)=1\end{equation}follows immediately from Proposition \ref{prop:neat2}, as does the fact that$$\sum_{n=2}^{\infty}\prod_{j=1}^{n-1}f_j\Big(1-\Big(\frac{1-p_j}{p_j}\Big)^{n-j}\Big)<\infty\implies{\bf Q}(K<\infty)=1$$our task is reduced to establishing the implication ${\bf Q}(K<\infty)=1\implies{\bf Q}(K=0)>0$.  Now recalling that \eqref{loneimp} implies that if ${\bf Q}(K<\infty)=1$ then $\exists\ L$ such that ${\bf Q}(V_L|\omega(L)=0)>0$, we find that ${\bf Q}(K=0)\geq\Big(\frac{1-p_0}{p_0}\Big)^L{\bf Q}(V_L|\omega(L)=0)>0$ (where $\Big(\frac{1-p_0}{p_0}\Big)^L$ is the probability that the frog starting at $x=0$ in the $\text{F}^+$ model ever hits the point $x=L$), thus completing the final step of the proof.
  \end{proof}
  
  \subsection{A simple proof of Gantert and Schmidt's result} \label{ss:GSpr}
  
  In order to demonstrate the utility of Theorem \ref{theorem:neat1}, this section is devoted to showing how it can be used to obtain a simple proof of the result from \cite{Gantert} described in the introduction.  The proof will be broken up into two parts.  While part 1 uses a method similar to Gantert and Schmidt's, part 2 employs a more novel approach which simplifies matters considerably.
  
  \medskip
  \noindent
  {\bf Part 1:}
  $\text{WTS: }\E[\text{log}^+\eta]=\infty\implies\text{recurrence}$
  
  \medskip
  \noindent
  Begin by defining the process $\left\{A_j\right\}$ where for every $j\in\mathbb{Z}/\left\{0\right\}\ A_j$ represents the number of distinct frogs originating at $x=j$ that ever hit the origin in the Gantert-Schmidt model.  Next we define the triple $(\Omega^*,\mathcal{F}^*,{\bf P}^*)$ where $\Omega^*$ represents the set of functions $\omega:\mathbb{Z}/\left\{0\right\}\rightarrow\mathbb{N}$ (i.e. the possible trajectories of $\left\{A_j\right\}$), $\mathcal{F}^*$ represents the $\sigma$-field on $\Omega^*$ generated by the finite dimensional sets, and ${\bf P}^*$ represents the probability measure induced on $(\Omega^*,\mathcal{F}^*)$ by the process $\left\{A_j\right\}$.  Additionally, denoting the two sided sequence $\left\{...,\eta_{-2},\eta_{-1},\eta_1,\eta_2,\dots\right\}$ that gives the number of sleeping frogs beginning at every nonzero vertex as $H$, we define (for every instance of $H$) the process $\left\{A^{(H)}_j\right\}$ in the same way as $\left\{A_j\right\}$, but where the number of sleeping frogs starting at each vertex is given by the terms of $H$.  As with $\left\{A_j\right\}$, each such process can be identified with a triple $(\Omega^*,\mathcal{F}^*,{\bf P}^*_H)$, where ${\bf P}^*_H$ represents the probability measure that $\left\{A^{(H)}_j\right\}$ induces on $(\Omega^*,\mathcal{F}^*)$ (the same $\sigma$-field referenced above).  Now since the activated frogs in this model all have nonzero leftward drift, this means all frogs that begin to the left of the origin are activated with probability $1$.  Hence, for $j\geq 1$ and $H=\left\{...,\eta_{-2},\eta_{-1},\eta_1,\eta_2,\dots\right\}$, we find that ${\bf P}^*_H(\omega(-j)>0)=1-\Big(1-\Big(\frac{1-p}{p}\Big)^j\Big)^{\eta_{-j}}$.  Now defining $U(\omega)=\#\left\{j\in\mathbb{Z}^+:\omega(-j)>0\right\}$, noting that the random variables $\omega(-j)$ are independent with respect to ${\bf P}^*_{H}$, and noting that if $\eta_{-j}\geq\Big(\frac{p}{1-p}\Big)^j$ then ${\bf P}^*_H(\omega(-j)>0)=1-\Big(1-\Big(\frac{1-p}{p}\Big)^j\Big)^{\eta_{-j}}\geq 1-e^{-1}>0$, we see that the implication\begin{equation}\label{ineqioimpeq}\left\{\eta_{-j}\geq\Big(\frac{p}{1-p}\Big)^j\ \text{i.o.}\right\}\implies{\bf P}^*_H(U=\infty)=1\end{equation}follows from B.C. II.  Furthermore, if we define $\Gamma=\left\{H\in(\eta_j)_{j\in\mathbb{Z}^*}:\eta_{-j}\geq\big(\frac{p}{1-p}\big)^j\ \text{i.o.}\right\}$ and let $\mu$ represent the probability measure associated with $(\eta_j)_{j\in\mathbb{Z}^*}$, then since$$\sum_{j=1}^{\infty}\Prob\Big(\eta\geq\Big(\frac{p}{1-p}\Big)^j\Big)=\sum_{j=1}^{\infty}\Prob\Big(\text{log}^+\eta\geq j\text{log}\Big(\frac{p}{1-p}\Big)\Big)\geq\sum_{j=1}^{\infty}\Prob\Big(\text{log}^+\eta\geq j\left\lceil{\text{log}\Big(\frac{p}{1-p}\Big)}\right\rceil\Big)$$ $$\geq\frac{1}{\left\lceil{\text{log}\Big(\frac{p}{1-p}\Big)}\right\rceil}\Big(\E[\text{log}^+\eta]-\left\lceil{\text{log}\Big(\frac{p}{1-p}\Big)}\right\rceil\Big)$$we find that another application of B.C. II gives the implication $\E[\text{log}^+\eta]=\infty\implies\mu(\Gamma)=1$.  Alongside \eqref{ineqioimpeq}, this establishes part 1.
  
  \medskip
  \noindent
  {\bf Part 2:}
  $\text{WTS:}\ \E[\text{log}^+\eta]<\infty\implies\text{transience}$
  
  \medskip
  \noindent
  Choose a constant $C$ such that $0<C<1$ and $C\cdot\frac{p}{1-p}>1$.  Noting that$$\sum_{j=1}^{\infty}\mu\Big(\eta_{-j}\geq C^j\Big(\frac{p}{1-p}\Big)^j\Big)=\sum_{j=1}^{\infty}\Prob\Big(\text{log}^+\eta\geq j\text{log}\Big(\frac{Cp}{1-p}\Big)\Big)\leq\frac{1}{\text{log}\Big(\frac{Cp}{1-p}\Big)}\E[\text{log}^+\eta]$$it follows from B.C. I that\begin{equation}\label{expimpz}\E[\text{log}^+\eta]<\infty\implies\mu\Big(\eta_{-j}\geq C^j\Big(\frac{p}{1-p}\Big)^j\ \text{i.o.}\Big)=0\end{equation}In addition, since for $j\geq 1$ we have ${\bf P}^*_H(\omega(-j)>0)=1-\Big(1-\Big(\frac{1-p}{p}\Big)^j\Big)^{\eta_{-j}}$ (see line preceding \eqref{ineqioimpeq}) and $$1-\Big(1-\Big(\frac{1-p}{p}\Big)^j\Big)^{C^j\big(\frac{p}{1-p}\big)^j}=(1+o(1))C^j\ \text{as}\ j\rightarrow\infty$$we find that if $\eta_{-j}\geq C^j\Big(\frac{p}{1-p}\Big)^j$ at only finitely many points, then $\sum_{j=1}^{\infty}{\bf P}^*_H(\omega(-j)>0)<\infty$.  Now coupling this with \eqref{expimpz} and employing B.C. I, we get (for $j\geq 1$)\begin{equation}\label{expfimpz}\E[\text{log}^+\eta]<\infty\implies{\bf P}^*(\omega(-j)>0\ \text{i.o.})=0\end{equation}
  
  Letting $\mathcal{A}=\sum_{j=1}^{\infty}\omega(-j)$, it follows from \eqref{expfimpz} that $\E[\text{log}^+\eta]<\infty\implies{\bf P}^*(\mathcal{A}<\infty)=1$.  If we now let $\mathcal{B}=\sum_{j=1}^{\infty}\omega(j)$, we find that in order to prove that $\E[\text{log}^+\eta]<\infty$ implies transience, it suffices to establish that for each $k$ with $0\leq k<\infty$ the following implication holds.\begin{equation}\label{fexpas}\E[\text{log}^+\eta]<\infty\implies{\bf P}^*(\mathcal{B}<\infty|\mathcal{A}=k)=1\end{equation}Now note that in terms of whether or not $\mathcal{B}=\infty$, the only relevant detail regarding the frogs beginning to the left of the origin is how far the one(s) that travels the furthest to the right of the origin gets.  Denoting this value as $\mathcal{C}$, if we assume ${\bf P}^*(\mathcal{B}=\infty)>0$, then there would have to exist $r\geq 0$ such that ${\bf P}^*(\mathcal{B}=\infty|\mathcal{C}=r)>0$.  Since the frog beginning at the origin reaches the point $x=r$ with positive probability, it would follow that ${\bf P}^*(\mathcal{B}=\infty|\mathcal{A}=0)>0$.  Hence, in order to establish \eqref{fexpas}, it suffices to establish the implication $\E[\text{log}^+\eta]<\infty\implies{\bf P}^*(\mathcal{B}<\infty|\mathcal{A}=0)=1$.
  
  The next step is to observe that $\big(\mathcal{B}|\mathcal{A}=0\big)$ has the same distribution as the number of distinct (initially sleeping) frogs that hit the origin in the non-homogeneous model on $\mathbb{Z}$ (in the case where $p_j=p$ for each $j\geq 0$ and the $X_j$'s are i.i.d. copies of $\eta$).  Using Theorem \ref{theorem:neat1}, it then follows that in order to establish that $\E[\text{log}^+\eta]<\infty$ implies transience, it is sufficient to establish the implication\begin{equation}\label{expfspinf}\E[\text{log}^+\eta]<\infty\implies\sum_{n=2}^{\infty}\prod_{j=1}^{n-1}f\Big(1-\Big(\frac{1-p}{p}\Big)^j\Big)=\infty\end{equation}(where $f$ represents the probability generating function of $\eta$).  Now noting that\begin{equation}\label{speqess}\sum_{n=2}^{\infty}\prod_{j=1}^{n-1}f\Big(1-\Big(\frac{1-p}{p}\Big)^j\Big)=\E\big[\sum_{n=2}^{\infty}e^{\sum_{j=1}^{n-1}\text{log}(1-(\frac{1-p}{p})^j)X_j}\big]\end{equation}we observe that because $\text{log}\big(1-\big(\frac{1-p}{p}\big)^j\big)=-(1+o(1))\big(\frac{1-p}{p}\big)^j$ as $j\rightarrow\infty$, it follows that if we have $0<C<1$ such that $\frac{Cp}{1-p}>1$ and $X_j\leq\Big(\frac{Cp}{1-p}\Big)^j$ for all but finitely many $j$, then$$\sum_{n=2}^{\infty}e^{\sum_{j=1}^{n-1}\text{log}(1-(\frac{1-p}{p})^j)X_j}=\infty$$When coupled with \eqref{expimpz} (where we replace $\eta_{-j}$ with $\eta_j$ on the right) and \eqref{speqess}, this establishes \eqref{expfspinf} which, as we saw, indicates that the left side of \eqref{expfspinf} implies transience, thus completing the proof.
  
  \section{Applications of Theorem \ref{theorem:neat1}}
  
  \subsection{Sharp conditions for the i.i.d. case} \label{ss:gengansch}
  
  Having shown in Section \ref{ss:GSpr} how Theorem \ref{theorem:neat1} can be used to obtain a concise proof of Gantert and Schmidt's result from \cite{Gantert}, this subsection is devoted to establishing a new result that involves a model similar to the one from \cite{Gantert}, but where the drifts of the individual frogs are dependent on where they originated (it will be assumed that no sleeping frogs reside to the left of the origin).  This result comes in the form of the following theorem.
  
  \begin{theorem}\label{theorem:neat3}
  For any version of the non-homogeneous frog model on $\mathbb{Z}$ for which the $X_j$'s are i.i.d. with $\E[X_1]<\infty$, $p_j=\frac{1}{2}+a_j$ with $g(j)=\frac{1}{a_j}$ being concave, and $d=\text{min}\left\{j:\Prob(X_1=j)>0\right\}$, the model is transient if and only if $\sum_{n=1}^{\infty}\frac{e^{-\frac{\mathcal{K}}{4a_n}}}{(a_n)^{d/2}}=\infty$ (where $f$ represents the generating function of $X_j$ and $\mathcal{K}=-\int_0^{\infty}\text{log}[f(1-e^{-x})]dx$).
  \end{theorem}
  
 \medskip
 \noindent $Remark\ 1.$ Note that $X_1$ having finite first moment (as stated in the theorem) gives us$$\E[X_1]<\infty\implies f'(1)=\E[X_1]<\infty\implies\text{log}[f(1-e^{-x})]=-q\cdot e^{-x}+o(e^{-x})\implies\mathcal{K}<\infty$$(where $q=f'(1)$).
 
 \medskip
 \noindent $Remark\ 2.$ One noteworthy (and immediate) consequence of Theorem \ref{theorem:neat3} is that for fixed $f$, $a_n=\frac{\mathcal{K}/4}{\text{log}n}$ (for all but finitely many $n$) represents a natural critical case in the sense that for $a_n=\frac{C}{\text{log}n}$ the model is transient if and only if $C\geq\mathcal{K}/4$.  An instance of particular significance is the case where $X_j=1\ \forall\ j$ (i.e. each positive integer point begins with exactly one sleeping frog).  Since in this scenario $f(x)=x$, we find that$$\mathcal{K}=\int_0^{\infty}|\text{log}(1-e^{-x})|dx=\int_0^{\infty}\sum_{n=1}^{\infty}\frac{e^{-nx}}{n}dx=\sum_{n=1}^{\infty}\int_0^{\infty}\frac{e^{-nx}}{n}dx=\sum_{n=1}^{\infty}\frac{1}{n^2}=\frac{\pi^2}{6}$$Hence, it follows that if $a_n=\frac{C}{\text{log}n}$, then the model is transient if and only if $C\geq\frac{\pi^2}{24}$, thus providing a new phase transition for the model from \cite{BMZ} that was mentioned in the introduction.
  
  \medskip
  \begin{proof}[Proof of Theorem \ref{theorem:neat3}]
   Given our result in Theorem \ref{theorem:neat1}, it follows that in order to establish this new result, it will suffice to show that\begin{equation}\label{equivpzez}\sum_{n=1}^{\infty}\frac{e^{-\frac{\mathcal{K}}{4a_n}}}{(a_n)^{d/2}}=\infty\iff\sum_{n=2}^{\infty}\prod_{j=1}^{n-1}f\Big(1-\Big(1-\frac{4a_{n-j}}{1+2a_{n-j}}\Big)^j\Big)=\infty\end{equation}(where the expression on the right in \eqref{equivpzez} is obtained by substituting $\frac{1}{2}+a_j$ for $p_j$ and switching $j$ and $n-j$ in \eqref{generalform}).  Furthermore, if we define $w_n=\frac{4a_n}{1+2a_n}$ and note that\begin{equation}\label{weqaeqratio}\frac{e^{-\frac{\mathcal{K}}{w_n}}}{(w_n)^{d/2}}\bigg/\frac{e^{-\frac{\mathcal{K}}{4a_n}}}{(a_n)^{d/2}}\rightarrow A e^{-\frac{\mathcal{K}}{2}}\ \text{as }n\rightarrow\infty\end{equation}(where $A=\underset{n\rightarrow\infty}{\text{lim}}\big(\frac{1+2a_n}{4}\big)^{d/2}$) we find that \eqref{equivpzez} is equivalent to the following:\begin{equation}\label{newetequiv}\sum_{n=1}^{\infty}\frac{e^{-\frac{\mathcal{K}}{w_n}}}{(w_n)^{d/2}}=\infty\iff\sum_{n=2}^{\infty}\prod_{j=1}^{n-1}f(1-(1-w_{n-j})^j)=\infty\end{equation}We'll first establish \eqref{equivpzez} (via \eqref{newetequiv}) under the condition that $a^{-1}_n$ is $O\big(\sqrt{n}\big)$ (see steps (i)-(iv)), following which we'll address the general case.
   
   \medskip
   \noindent
   (i) $\sum_{n=2}^{\infty}\prod_{j=1}^{n-1}f(1-(1-w_n)^j)=\infty\iff\sum_{n=2}^{\infty}\prod_{j=1}^{n-1}f(1-(1-w_{n-j})^j)=\infty$
   
   \medskip
   \noindent
   Since $a_n$ is decreasing this means $w_n$ is as well, from which it follows that$$\prod_{j=1}^{n-1}f(1-(1-w_{n-j})^j)\geq\prod_{j=1}^{n-1}f(1-(1-w_n)^j)$$for all $n$.  Hence, in order to establish (i) it suffices to show that\begin{equation}\label{1stlogdiff}\text{limsup}\sum_{j=1}^{n-1}\text{log}[f(1-(1-w_{n-j})^j)]-\text{log}[f(1-(1-w_n)^j)]<\infty\end{equation}Rewriting the expression in \eqref{1stlogdiff} (see below) we now get the inequality\begin{equation}\label{2ndlogdiff}\text{limsup}\sum_{j=1}^{n-1}\text{log}[f(1-(1-w_n)^j+\big((1-w_n)^j-(1-w_{n-j})^j\big))]-\text{log}[f(1-(1-w_n)^j)]\end{equation}$$\leq\text{limsup}\sum_{j=1}^{n-1}\text{log}[f(1-(1-w_n)^j+\big((j\cdot(w_{n-j}-w_n))\wedge (1-w_n)\big)\cdot(1-w_n)^{j-1})]-\text{log}[f(1-(1-w_n)^j)]$$Since $a^{-1}_n$ being $O\big(\sqrt{n}\big)$ implies $w^{-1}_n$ is as well, this means that the larger of the two expressions in \eqref{2ndlogdiff} is equal to the smaller expression in the following inequality.\begin{equation}\label{logdiffslog}\underset{n-1>\frac{1}{w_n}}{\text{limsup}}\sum_{j=1}^{n-1}\text{log}[f(1-(1-w_n)^j+\big((j\cdot(w_{n-j}-w_n))\wedge (1-w_n)\big)\cdot(1-w_n)^{j-1})]-\text{log}[f(1-(1-w_n)^j)]\end{equation}$$\leq\underset{n-1>\frac{1}{w_n}}{\text{limsup}}\sum_{j\leq\frac{1}{w_n}}\text{log}[f(1-(1-w_n)^j+\big((j\cdot(w_{n-j}-w_n))\wedge (1-w_n)\big)\cdot(1-w_n)^{j-1})]-\text{log}[f(1-(1-w_n)^j)]$$ $$+\underset{n-1>\frac{1}{w_n}}{\text{limsup}}\sum_{\frac{1}{w_n}<j\leq n-1}\frac{q\cdot(w_{n-j}-w_n)\cdot j\cdot(1-w_n)^{j-1}}{f(1-(1-w_n)^{\frac{1}{w_n}})}$$(recall that $q=f'(1)$).  The second term to the right of the inequality in \eqref{logdiffslog} can now be bounded above by\begin{equation}\label{2ndtbnds}\frac{q}{f(1-e^{-1})}\underset{n-1>\frac{1}{w_n}}{\text{limsup}}\sum_{\frac{1}{w_n}<j\leq n-1}(w_{n-j}-w_n)\cdot j\cdot(1-w_n)^{j-1}\end{equation}$$\leq\frac{q}{f(1-e^{-1})}\underset{n-1>\frac{1}{w_n}}{\text{limsup}}\sum_{\frac{1}{w_n}<j\leq n-1}\Bigg(1-\frac{w_n}{w_{n-j}}\Bigg)\cdot\frac{w_{n-j}}{w_n}\cdot w_n\cdot j\cdot e^{-w_n (j-1)}$$ $$\leq\frac{q\cdot e}{f(1-e^{-1})}\underset{n-1>\frac{1}{w_n}}{\text{limsup}}\sum_{\frac{1}{w_n}<j\leq n-1}\frac{1}{w_n\cdot(n-j)}\cdot(w_n\cdot j)^2\cdot e^{-w_n j}$$(where the final inequality in \eqref{2ndtbnds} follows from the fact that $\frac{w_{n-j}}{w_n}\leq\frac{n}{n-j}$, which follows from the concavity of $\frac{1}{w_n}$, which in turn follows from the concavity of $\frac{1}{a_n}$).  Next we bound the final term in \eqref{2ndtbnds} by\begin{equation}\label{divbli}\frac{q\cdot e}{f(1-e^{-1})\cdot\text{liminf}\ nw_n^2}\ \underset{n-1>\frac{1}{w_n}}{\text{limsup}}\sum_{\frac{1}{w_n}<j\leq n-1}\frac{1}{1-\frac{j}{n}}\cdot w_n\cdot(w_n j)^2\cdot e^{-w_n j}\end{equation}$$\leq\frac{q\cdot e}{f(1-e^{-1})\cdot\text{liminf}\ nw_n^2}\ \underset{n-1>\frac{1}{w_n}}{\text{limsup}}\sum_{\frac{1}{w_n}<j\leq n-1}w_n\cdot(w_n j)^2\cdot e^{\frac{-3w_n j}{4}}$$(with the last inequality following from the fact, implied by $w^{-1}_n$ being $O\big(\sqrt{n}\big)$, that for sufficiently large $n$ we have $\frac{1}{1-\frac{j}{n}}\leq e^{n^{-\frac{2}{3}}j}\leq e^{\frac{w_n j}{4}}$ for $1\leq j\leq n-1$).  Finally, comparing the sum inside the larger term in \eqref{divbli} to the integral of $x^2 e^{-\frac{3}{4}x}$ on $[\lceil{\frac{1}{w_n}}\rceil\cdot w_n,\ n\cdot w_n]$, we see that there must exist $K<\infty$ (independent of $n$) such that the sum is bounded above by $\int_0^{\infty}x^2 e^{\frac{-3}{4}x}dx+K$.  Combining this with $w_n^{-1}$ being $O(\sqrt{n})$ then implies that the bottom expression in \eqref{divbli} is finite which, coupled with \eqref{2ndtbnds} and \eqref{divbli}, now establishes that the second term on the right of the inequality in \eqref{logdiffslog} is finite.
   
   To complete the proof of (i), it now just needs to be shown that the first term on the right of the inequality in \eqref{logdiffslog} is finite as well.  Now because for any probability generating function of a non negative integer valued random variable with finite mean $\frac{f'(x)}{f(x)}$ is $O\big(\frac{1}{x}\big)$, this means there must exist a constant $C<\infty$ such that $\frac{f'(x)}{f(x)}\leq\frac{C}{x}\ \forall\ x\in\ (0,1]$, from which it follows that the term in question is bounded above by\begin{equation}\label{bnd2ndtis}\underset{n-1>\frac{1}{w_n}}{\text{limsup}}\sum_{j\leq\frac{1}{w_n}}\frac{C\cdot(w_{n-j}-w_n)\cdot j\cdot(1-w_n)^{j-1}}{1-(1-w_n)^j}\end{equation}Next noting that for $x\in(0,1]$ and $m\in\mathbb{Z}^+$ we have $\frac{1-(1-x)^m}{mx}=\frac{1}{m}\big(1+(1-x)+\dots +(1-x)^{m-1}\big)\geq(1-x)^{m-1}$, it follows that \eqref{bnd2ndtis} can be bounded above by$$C\cdot\underset{n-1>\frac{1}{w_n}}{\text{limsup}}\sum_{j\leq\frac{1}{w_n}}\frac{(w_{n-j}-w_n)}{w_n}=C\cdot\underset{n-1>\frac{1}{w_n}}{\text{limsup}}\sum_{j\leq\frac{1}{w_n}}\Bigg(\frac{1}{w_n}-\frac{1}{w_{n-j}}\Bigg)\cdot w_{n-j}$$On account of the concavity of $\frac{1}{w_n}$, this last expression can itself be bounded above by$$C\cdot\underset{n-1>\frac{1}{w_n}}{\text{limsup}}\sum_{j\leq\frac{1}{w_n}}\frac{j}{n}\cdot\frac{1}{w_n}\cdot\frac{n}{n-j}\cdot w_n=C\cdot\underset{n-1>\frac{1}{w_n}}{\text{limsup}}\sum_{j\leq\frac{1}{w_n}}\frac{j}{n-j}=\frac{C}{2}\cdot\underset{n-1>\frac{1}{w_n}}{\text{limsup}}\frac{1}{n\cdot w_n^2}<\infty$$(where the second equality along with the finiteness of the last term both follow from the fact that $w^{-1}_n$ is $O\big(\sqrt{n}\big)$).  Hence, this establishes that \eqref{bnd2ndtis}, as well as the first term to the right of the inequality in \eqref{logdiffslog}, is finite.  Now if this is combined with the finiteness of the second expression to the right of the inequality in \eqref{logdiffslog}, along with the inequality in \eqref{2ndlogdiff}, we see that \eqref{1stlogdiff} follows, thus completing the proof of (i).
   
   \medskip
   \noindent
   (ii) $\sum_{n=2}^{\infty}\prod_{j=1}^{n-1}f(1-e^{-w_n\cdot j})=\infty\iff\sum_{n=2}^{\infty}\prod_{j=1}^{n-1}f(1-(1-w_n)^j)=\infty$
   
   \medskip
   \noindent
   Because we know that$$1-w_n\leq e^{-w_n}\implies\sum_{n=2}^{\infty}\prod_{j=1}^{n-1}f(1-e^{-w_n\cdot j})\leq\sum_{n=2}^{\infty}\prod_{j=1}^{n-1}f(1-(1-w_n)^j)$$it follows that in order to establish (ii), it suffices to show (much like in the case of (i)) that\begin{equation}\label{logdiffwexp}\text{limsup}\sum_{j=1}^{n-1}\text{log}[f(1-(1-w_n)^j)]-\text{log}[f(1-e^{-w_n\cdot j})]<\infty\end{equation}Defining $C_n=\frac{e^{-w_n}-(1-w_n)}{w_n^2}$, we have the following string of inequalities (where the expression on the first line equals the expression in \eqref{logdiffwexp}, and with $S(n,j)$ representing the summand on the second line).\begin{equation}\label{strinequlogdiffe}\text{limsup}\sum_{j=1}^{n-1}\text{log}[f(1-e^{-w_n\cdot j}+\big((1-w_n+C_n w_n^2)^j-(1-w_n)^j\big)]-\text{log}[f(1-e^{-w_n\cdot j})]\end{equation}$$\leq\text{limsup}\sum_{j=1}^{n-1}\text{log}[f(1-e^{-w_n\cdot j}+\big((j\cdot C_n\cdot w_n^2)\wedge (1-w_n)\big)\cdot e^{-w_n(j-1)})]-\text{log}[f(1-e^{-w_n\cdot j})]$$ $$\leq\underset{n-1>\frac{1}{w_n}}{\text{limsup}}\sum_{j\leq\frac{1}{w_n}}S(n,j)+\underset{n-1>\frac{1}{w_n}}{\text{limsup}}\sum_{\frac{1}{w_n}<j\leq n-1}S(n,j)$$If we can show that both of the expressions on the last line of \eqref{strinequlogdiffe} are finite, then \eqref{logdiffwexp} will immediately follow.  Beginning with the first expression, observe that if we use the fact (referenced in the proof of (i)) that there must exist $C<\infty$ such that $\frac{f'(x)}{f(x)}\leq\frac{C}{x}\ \forall\ x\in(0,1]$, then we can obtain the string of inequalities\begin{equation}\label{strinequscj}\underset{n-1>\frac{1}{w_n}}{\text{limsup}}\sum_{j\leq\frac{1}{w_n}}S(n,j)\leq\underset{n-1>\frac{1}{w_n}}{\text{limsup}}\sum_{j\leq\frac{1}{w_n}}\frac{C\cdot C_n\cdot j\cdot w_n^2\cdot e^{-w_n(j-1)}}{1-e^{-w_n\cdot j}}\leq\frac{C}{2}\underset{n-1>\frac{1}{w_n}}{\text{limsup}}\sum_{j\leq\frac{1}{w_n}}\frac{j\cdot w_n^2}{1-e^{-w_n\cdot j}}\end{equation}(where the second inequality follows from the fact that $C_n\leq\frac{1}{2}\ \forall\ n$).  Now using the fact that$$1-e^{-w_n\cdot j}=\big(1-e^{-w_n}\big)\cdot\big(1+e^{-w_n}+\dots+\big(e^{-w_n}\big)^{j-1}\big)\geq j\cdot\big(1-e^{-w_n}\big)\cdot\big(e^{-w_n}\big)^{j-1}$$and that $\frac{1-e^{-w_n}}{w_n}\geq 1-e^{-1}$ (since $0<w_n<1\ \forall\ n)$, we find that the expression on the right in \eqref{strinequscj} is bounded above by$$\frac{C}{2(1-e^{-1})}\underset{n-1>\frac{1}{w_n}}{\text{limsup}}\sum_{j\leq\frac{1}{w_n}}\frac{j\cdot w_n^2}{j\cdot w_n\cdot e^{-1}}=\frac{C\cdot e}{2(1-e^{-1})}\underset{n-1>\frac{1}{w_n}}{\text{limsup}}\sum_{j\leq\frac{1}{w_n}}w_n\leq\frac{C\cdot e}{2(1-e^{-1})}<\infty$$thus establishing that the first sum on the last line of \eqref{strinequlogdiffe} is finite.
   
   In order to establish \eqref{logdiffwexp}, and thus complete the proof of (ii), it only remains to show that the second sum on the last line of \eqref{strinequlogdiffe} is finite as well.  We accomplish this via the following string of inequalities:$$\underset{n-1>\frac{1}{w_n}}{\text{limsup}}\sum_{\frac{1}{w_n}<j\leq n-1}S(n,j)\leq\frac{C}{2(1-e^{-1})}\underset{n-1>\frac{1}{w_n}}{\text{limsup}}\sum_{\frac{1}{w_n}<j<\infty}w_n\cdot(w_n j)\cdot e^{-w_n(j-1)}\leq\frac{C\cdot e}{2(1-e^{-1})}\int_0^{\infty}x\cdot e^{-x}dx +K$$(where the first inequality follows from the same argument used in \eqref{strinequscj}).  Hence, the proof of (ii) is complete.
   
   \medskip
   \noindent
   (iii) $\sum_{n=2}^{\infty}\prod_{j=1}^{\infty}f(1-e^{-w_n\cdot j})=\infty\iff\sum_{n=2}^{\infty}\prod_{j=1}^{n-1}f(1-e^{-w_n\cdot j})=\infty$
   
   \medskip
   \noindent
   Since one direction is immediate, we're left with just having to show that\begin{equation}\label{boundtailsl}\text{limsup}\sum_{j=n}^{\infty}-\text{log}[f(1-e^{-w_n\cdot j})]<\infty\end{equation}Observing that$$\text{limsup}\sum_{j=n}^{\infty}-\text{log}[f(1-e^{-w_n\cdot j})]=\text{limsup}\frac{1}{w_n}\sum_{j=n}^{\infty}-w_n\text{log}[f(1-e^{-w_n\cdot j})]$$ $$\leq\text{limsup}\frac{1}{w_n}\int_{(n-1)\cdot w_n}^{\infty}-\text{log}[f(1-e^{-x})]dx$$we find that, as a consequence of the fact that $f'(1)=q<\infty$ and $w^{-1}_n$ is $O\big(\sqrt{n}\big)$, we have$$\text{limsup}\frac{1}{w_n}\int_{(n-1)\cdot w_n}^{\infty}-\text{log}[f(1-e^{-x})]dx=\text{limsup}\frac{1}{w_n}\int_{(n-1)\cdot w_n}^{\infty}q\cdot e^{-x}dx=\text{limsup}\frac{1}{w_n}\cdot q\cdot e^{-(n-1)\cdot w_n}$$ $$\leq\text{limsup}\frac{\sqrt{n}}{\sqrt{l}}\cdot q\cdot e^{-\frac{n-1}{n}\cdot\sqrt{n}\cdot\sqrt{l}}=0$$(where $l$ denotes the value of $\text{liminf}\ n\cdot w_n^2$).  Hence, this establishes \eqref{boundtailsl}, thus completing the proof of (iii).
   
   \medskip
   \noindent
   (iv) $\sum_{n=1}^{\infty}\frac{e^{\frac{-\mathcal{K}}{w_n}}}{(w_n)^{d/2}}=\infty\iff\sum_{n=2}^{\infty}\prod_{j=1}^{\infty}f(1-e^{-w_n\cdot j})=\infty$
   
   \medskip
   \noindent
   Denoting $c_d=\Prob(X_1=d)$ (recall $d=\text{min}\left\{j:\Prob(X_1=j)>0\right\}$), observe that\begin{equation}\label{approxderivlogf}\frac{d(\text{log}[f(x)])}{dx}=\frac{f'(x)}{f(x)}=\frac{d c_d+(d+1)c_{d+1}x+\dots}{c_d x+c_{d+1}x^2+\dots}=\frac{d}{x}\cdot\frac{1+\frac{d+1}{d}\frac{c_{d+1}}{c_d}x+\dots}{1+\frac{c_{d+1}}{c_d}x+\dots}=\frac{d}{x}+O(1)\end{equation}Now we want to approximate\begin{equation}\label{splitdiffintsum}\frac{-\mathcal{K}}{w_n}-\text{log}\Big[\prod_{j=1}^{\infty}f(1-e^{-w_n\cdot j})\Big]=\frac{1}{w_n}\int_0^{\lfloor{\frac{1}{w_n}}\rfloor w_n}\text{log}[f(1-e^{-x})]dx-\frac{1}{w_n}\sum_{j=1}^{\lfloor{\frac{1}{w_n}}\rfloor}w_n\cdot\text{log}[f(1-e^{-w_n\cdot j})]\end{equation} $$+\frac{1}{w_n}\int_{\lfloor{\frac{1}{w_n}}\rfloor w_n}^{\infty}\text{log}[f(1-e^{-x})]dx-\frac{1}{w_n}\sum_{\lceil{\frac{1}{w_n}}\rceil}^{\infty}w_n\cdot\text{log}[f(1-e^{-w_n\cdot j})]$$within an order of $O(1)$.  First noting that the expression on the second line of \eqref{splitdiffintsum} is $O(1)$ as $n\rightarrow\infty$ (this follows from the fact that it is bounded above by $0$ and below by $\text{log}\Big[f\Big(1-e^{-w_n\cdot\lfloor{\frac{1}{w_n}\rfloor}}\Big)\Big]$), we see that our task is reduced to approximating\begin{equation}\label{diffsuminteqsumint}\frac{1}{w_n}\int_0^{\lfloor{\frac{1}{w_n}}\rfloor w_n}\text{log}[f(1-e^{-x})]dx-\frac{1}{w_n}\sum_{j=1}^{\lfloor{\frac{1}{w_n}}\rfloor}w_n\cdot\text{log}[f(1-e^{-w_n\cdot j})]\end{equation} $$=\frac{1}{w_n}\sum_{j=2}^{\lfloor{\frac{1}{w_n}}\rfloor}\int_0^{w_n}\text{log}[f(1-e^{-(w_n\cdot j-t)})]-\text{log}[f(1-e^{-w_n\cdot j})]dt+O(1)$$(where the $O(1)$ term represents $\frac{1}{w_n}\int_0^{w_n}\text{log}[f(1-e^{-x})]dx-\text{log}[f(1-e^{-w_n})]$).  Using \eqref{approxderivlogf}, we then find that the integrand in the bottom expression equals$$-\int_{1-e^{-(w_n\cdot j-t)}}^{1-e^{-w_n\cdot j}}\frac{d}{x}+O(1)dx=d\text{log}\Big[\frac{1-e^{-(w_n\cdot j-t)}}{1-e^{-w_n\cdot j}}\Big]+O(e^{-(w_n\cdot j-t)}-e^{-w_n\cdot j})=d\text{log}\Big[\frac{1-e^{-(w_n\cdot j-t)}}{1-e^{-w_n\cdot j}}\Big]+O(t)$$ $$=d\text{log}\Bigg[1+\frac{e^{-w_n\cdot j}(1-e^t)}{1-e^{-w_n\cdot j}}\Bigg]+O(t)=d\text{log}\Big[1-\frac{t}{w_n\cdot j}+O(t)\Big]+O(t)=d\text{log}\Big[1-\frac{t}{w_n\cdot j}\Big]+O(t)$$(with the final equality following from the fact that $n\geq 2\implies 1-\frac{t}{w_n\cdot j}\geq\frac{1}{2}>0\ \forall\ t$).  Plugging this back into the expression on the second line of \eqref{diffsuminteqsumint} now gives$$\frac{1}{w_n}\sum_{j=2}^{\lfloor{\frac{1}{w_n}}\rfloor}\int_0^{w_n}d\text{log}\Big[1-\frac{t}{w_n\cdot j}\Big]+O(t)dt=\frac{d}{w_n}\sum_{j=2}^{\lfloor{\frac{1}{w_n}}\rfloor}-(j-1)\cdot w_n\cdot\text{log}\Big[1-\frac{1}{j}\Big]-w_n+O(w_n^2)$$ $$=d\sum_{j=2}^{\lfloor{\frac{1}{w_n}}\rfloor}-(j-1)\cdot\text{log}\Big[1-\frac{1}{j}\Big]-1+O(w_n)=d\sum_{j=2}^{\lfloor{\frac{1}{w_n}}\rfloor}\frac{-1}{2j}+O\Big(\frac{1}{j^2}\Big)+O(w_n)=\frac{-d}{2}\text{log}\Big[\frac{1}{w_n}\Big]+O(1)$$(where the $O(t)$ expressions above are to be interpreted as meaning that the absolute value of the term in question is bounded above by $ct$ for some $c<\infty$, where $c$ is independent of $n$ as well as $t$).  Looking back now at the first line of \eqref{splitdiffintsum}, we find that$$\frac{-\mathcal{K}}{w_n}-\text{log}\Big[\prod_{j=1}^{\infty}f(1-e^{-w_n\cdot j})\Big]=\frac{-d}{2}\text{log}\Big[\frac{1}{w_n}\Big]+O(1)\implies C_1\cdot\frac{e^{\frac{-\mathcal{K}}{w_n}}}{(w_n)^{d/2}}\leq\prod_{j=1}^{\infty}f(1-e^{-w_n\cdot j})\leq C_2\cdot\frac{e^{\frac{-\mathcal{K}}{w_n}}}{(w_n)^{d/2}}$$(for some $C_1,C_2$ independent of $n$ with $0<C_1<C_2<\infty$), thus completing the proof of (iv).
   
   \bigskip
   \noindent
   Having now established \eqref{equivpzez} via (i)-(iv) when $a^{-1}_n$ is $O\big(\sqrt{n}\big)$, our final task is to address the general case.  To do this, we first note that because the proof of (iv) does not use that $a^{-1}_n$ is $O\big(\sqrt{n}\big)$, it follows that it continues to hold without this assumption.  Coupling this with \eqref{weqaeqratio}, along with the fact that$$\sum_{n=2}^{\infty}\prod_{j=1}^{\infty}f(1-e^{-w_n\cdot j})\leq\sum_{n=2}^{\infty}\prod_{j=1}^{n-1}f(1-e^{-w_n\cdot j})\leq\sum_{n=2}^{\infty}\prod_{j=1}^{n-1}f(1-(1-w_n)^j)\leq\sum_{n=2}^{\infty}\prod_{j=1}^{n-1}f(1-(1-w_{n-j})^j)$$we find that the implication going from left to right in \eqref{equivpzez} holds regardless of whether or not $a^{-1}_n$ is $O\big(\sqrt{n}\big)$.  Hence, to complete the proof of the theorem we simply need to show that when $a^{-1}_n$ is not $O\big(\sqrt{n}\big)$, finiteness of the expression on the left side of \eqref{equivpzez}, still implies finiteness of the expression on the right.
   
   If we define the sequence $\tilde{a}_n$ so that$$\frac{1}{\tilde{a}_n}=\left\{\begin{array}{ll}\frac{1}{a_n} &\text{if }\frac{1}{a_n}<3\sqrt{n}\\3\sqrt{n} &\text{otherwise}\end{array}\right.$$it then follows that $\frac{1}{\tilde{a}_n}$ is concave and $O\big(\sqrt{n}\big)$ (also note $\frac{1}{2}<\frac{1}{2}+\tilde{a}_n<1$ still holds).  In addition, since we're assuming that the expression on the left side of \eqref{equivpzez} is finite, this means$$\sum_{n=1}^{\infty}\frac{e^{\frac{-\mathcal{K}}{4\tilde{a}_n}}}{(\tilde{a}_n)^{d/2}}\leq\sum_{n=1}^{\infty}\frac{e^{\frac{-\mathcal{K}}{4a_n}}}{(a_n)^{d/2}}+\sum_{n=1}^{\infty}e^{\frac{-\mathcal{K}\cdot 3\sqrt{n}}{4}}\cdot 3^{d/2}\cdot n^{d/4}<\infty$$Hence, the proof of \eqref{equivpzez}, for the case where $a^{-1}_n$ is $O\big(\sqrt{n}\big)$, implies that$$\sum_{n=2}^{\infty}\prod_{j=1}^{n-1}f\Big(1-\Big(1-\frac{4\tilde{a}_{n-j}}{1+2\tilde{a}_{n-j}}\Big)^j\Big)<\infty$$Coupling this with the fact that $a_n\leq\tilde{a}_n$, we can now conclude that$$\sum_{n=2}^{\infty}\prod_{j=1}^{n-1}f\Big(1-\Big(1-\frac{4a_{n-j}}{1+2a_{n-j}}\Big)^j\Big)<\infty$$which, along with the argument in the previous paragraph, establishes that \eqref{equivpzez} continues to hold when $a^{-1}_n$ is not $O\big(\sqrt{n}\big)$.  Hence, the proof of the theorem is complete.
   \end{proof}
  
  \subsection{Sharp conditions for the Poiss($\lambda_j$) scenario} \label{ss:genJR}
  
  In this section we'll address the final model discussed in the introduction (see \cite{Rosenberg}), establishing sharp conditions for the case where the drift values of individual frogs are dependent on where they originate.  Our result is as follows.
  
  \begin{theorem}\label{theorem:neat4}
  For $X_j=\text{Poiss}(\lambda_j)$ and $p_j=\frac{1}{2}+a_j$ (with the sequences $\frac{1}{a_j}$ and $\lambda_j$ both being concave), the nonhomogeneous frog model on $\mathbb{Z}$ is transient if and only if\begin{equation}\label{JRmodeltc}\sum_{n=1}^{\infty}e^{-\lambda_n\big(\frac{1}{4a_n}-\frac{1}{2}\big)}=\infty\end{equation}
  \end{theorem}
  
  \begin{proof}
  Since $\text{Poiss}(\lambda_j)$ has generating function $e^{\lambda_j(x-1)}$, applying Theorem \ref{theorem:neat1} reduces our task to showing that$$\sum_{n=1}^{\infty}e^{-\lambda_n\big(\frac{1}{4a_n}-\frac{1}{2}\big)}=\infty\iff\sum_{n=2}^{\infty}e^{-\sum_{j=1}^{n-1}\lambda_{n-j}\big(1-\frac{4a_{n-j}}{1+2a_{n-j}}\big)^j}=\infty$$Noting also that$$\sum_{n=2}^{\infty}e^{-\lambda_n\big(\frac{1}{4a_n}-\frac{1}{2}\big)}=\sum_{n=2}^{\infty}e^{-\sum_{j=1}^{\infty}\lambda_n\big(1-\frac{4a_n}{1+2a_n}\big)^j}\leq\sum_{n=2}^{\infty}e^{-\sum_{j=1}^{n-1}\lambda_n\big(1-\frac{4a_n}{1+2a_n}\big)^j}\leq\sum_{n=2}^{\infty}e^{-\sum_{j=1}^{n-1}\lambda_{n-j}\big(1-\frac{4a_{n-j}}{1+2a_{n-j}}\big)^j}$$we see that it will in fact suffice to establish the implication\begin{equation}\label{impfJRmodel}\sum_{n=1}^{\infty}e^{-\lambda_n\big(\frac{1}{4a_n}-\frac{1}{2}\big)}<\infty\implies\sum_{n=2}^{\infty}e^{-\sum_{j=1}^{n-1}\lambda_{n-j}\big(1-\frac{4a_{n-j}}{1+2a_{n-j}}\big)^j}<\infty\end{equation}To do this we'll begin by proving \eqref{impfJRmodel} for the case where $\lambda_n$ and $a^{-1}_n$ are both $O\big(n^{1/3}\big)$.  Much like with the proof of Theorem \ref{theorem:neat3}, this will be accomplished by showing that\begin{equation}\label{bndlsJRmodel}\text{limsup}\ \lambda_n\Big(\frac{1}{4a_n}-\frac{1}{2}\Big)-\sum_{j=1}^{n-1}\lambda_{n-j}\Big(1-\frac{4a_{n-j}}{1+2a_{n-j}}\Big)^j<\infty\end{equation}
  
  As a first step towards establishing \eqref{bndlsJRmodel}, we note the following string of inequalities (with $\epsilon_j$ denoting $\frac{a_j}{1+2a_j}$).\begin{equation}\label{bndls1stJRm}\text{limsup}\sum_{j=1}^{n-1}\lambda_n(1-4\epsilon_{n-j})^j-\sum_{j=1}^{n-1}\lambda_{n-j}(1-4\epsilon_{n-j})^j=\text{limsup}\sum_{j=1}^{n-1}(\lambda_n-\lambda_{n-j})\cdot(1-4\epsilon_{n-j})^j\end{equation}$$\leq\text{limsup}\sum_{j=1}^{n-1}(\lambda_n-\lambda_{n-j})\cdot(1-4\epsilon_n)^j\leq\text{limsup}\sum_{j=1}^{n-1}\frac{j}{n}\cdot\lambda_n\cdot e^{-4j\epsilon_n}\leq\text{limsup}\frac{\lambda_n/\epsilon^2_n}{n}\sum_{j=1}^{\infty}\epsilon_n\cdot(\epsilon_n j)\cdot e^{-4j\epsilon_n}<\infty$$(where the inequality between the first two sums on the second line follows from the fact that $\lambda_j$ is concave and $(1-4\epsilon_n)^j\leq e^{-4j\epsilon_n}$, and the finiteness of the last expression derives from the fact that $\lambda_n$ and $\epsilon^{-1}_n$ are both $O\big(n^{1/3}\big)$, along with the fact that the sum is bounded above by $\int_0^{\infty}x e^{-4x}dx+K$ for some $K<\infty$).  Next, we present another string of inequalities as shown.\begin{equation}\label{bndls2ndJRm}\text{limsup}\sum_{j=1}^{n-1}\lambda_n(1-4\epsilon_n)^j-\sum_{j=1}^{n-1}\lambda_n(1-4\epsilon_{n-j})^j=\text{limsup}\lambda_n\sum_{j=1}^{n-1}(1-4\epsilon_n)^j-(1-4\epsilon_{n-j})^j\end{equation}$$\leq\text{limsup}\ 4\lambda_n\sum_{j=1}^{n-1}(\epsilon_{n-j}-\epsilon_n)\cdot j\cdot(1-4\epsilon_n)^{j-1}=\text{limsup}\ 4\lambda_n\sum_{j=1}^{n-1}(\epsilon^{-1}_n-\epsilon^{-1}_{n-j})\cdot\epsilon_n\epsilon_{n-j}\cdot j\cdot(1-4\epsilon_n)^{j-1}$$Because $\epsilon^{-1}_n$ is concave (since it equals $a^{-1}_n+2$), it follows that the expression on the second line of \eqref{bndls2ndJRm} is less than or equal to$$\text{limsup}\ 4\lambda_n\sum_{j=1}^{n-1}\frac{j}{n}\cdot\epsilon_{n-j}\cdot j\cdot(1-4\epsilon_n)^{j-1}\leq\text{limsup}\ 4\lambda_n\sum_{j=1}^{n-1}\epsilon_{n-j}\cdot\frac{j^2}{n}\cdot e^{-4(j-1)\epsilon_n}\leq\text{limsup}\ 4e\lambda_n\sum_{j=1}^{n-1}\frac{\epsilon_n}{1-\frac{j}{n}}\cdot\frac{j^2}{n}\cdot e^{-4j\epsilon_n}$$ $$=\text{limsup}\ \frac{4e\lambda_n/\epsilon^2_n}{n}\sum_{j=1}^{n-1}\frac{1}{1-\frac{j}{n}}\cdot(j\epsilon_n)^2\cdot e^{-4j\epsilon_n}\epsilon_n\leq\text{limsup}\ \frac{4e\lambda_n/\epsilon^2_n}{n}\sum_{j=1}^{\infty}\epsilon_n\cdot(\epsilon_n j)^2\cdot e^{-3j\epsilon_n}<\infty$$(where the inequality on the second line follows from the fact that for sufficiently large $n$ we have $\frac{1}{1-\frac{j}{n}}<e^{j\epsilon_n}$ for all $j$ with $1\leq j<n$, and where the finiteness of the last term follows from $\lambda_n$ and $\epsilon^{-1}_n$ both being $O\big(n^{1/3}\big)$, along with the fact that the sum is once again bounded above by $\int_0^{\infty}x^2 e^{-3x}dx+K$ for some $K<\infty$).  Combining this last string of inequalities with \eqref{bndls2ndJRm}, we see that\begin{equation}\label{comb2strgs}\text{limsup}\sum_{j=1}^{n-1}\lambda_n(1-4\epsilon_n)^j-\sum_{j=1}^{n-1}\lambda_n(1-4\epsilon_{n-j})^j<\infty\end{equation}Finally, we observe that\begin{equation}\label{bndls3rdJRm}\text{limsup}\ \lambda_n\Big(\frac{1}{4a_n}-\frac{1}{2}\Big)-\sum_{j=1}^{n-1}\lambda_n(1-4\epsilon_n)^j=\text{limsup}\sum_{j=1}^{\infty}\lambda_n(1-4\epsilon_n)^j-\sum_{j=1}^{n-1}\lambda_n(1-4\epsilon_n)^j\end{equation}$$=\text{limsup}\sum_{j=n}^{\infty}\lambda_n(1-4\epsilon_n)^j=\text{limsup}\lambda_n\cdot\frac{(1-4\epsilon_n)^n}{4\epsilon_n}\leq\text{limsup}\frac{\lambda_n}{4\epsilon_n}\cdot e^{-4n\epsilon_n}=0$$(where the last equality again follows from $\lambda_n$ and $\epsilon^{-1}_n$ both being $O\big(n^{1/3}\big)$).  Now putting \eqref{bndls1stJRm}, \eqref{comb2strgs}, and \eqref{bndls3rdJRm} together, we see that \eqref{bndlsJRmodel} (and therefore \eqref{impfJRmodel}) does indeed hold if $\lambda_n$ and $a^{-1}_n$ are $O\big(n^{1/3}\big)$.
  
  To complete the proof of the theorem, \eqref{impfJRmodel} now just needs to be proven for the general case (i.e. without the condition that $\lambda_n$ and $a^{-1}_n$ are $O\big(n^{1/3})$).  To do this we begin by defining $\tilde{\lambda}_n$ and $\tilde{a}_n$ as$$\tilde{\lambda}_n=\left\{\begin{array}{ll}\lambda_n &\text{if }\lambda_n< n^{1/3}\\ n^{1/3} &\text{otherwise}\end{array}\right.$$and$$\frac{1}{\tilde{a}_n}=\left\{\begin{array}{ll}\frac{1}{a_n} &\text{if }\frac{1}{a_n}<3n^{1/3}\\3n^{1/3} &\text{otherwise}\end{array}\right.$$(again the coefficient $3$ has been chosen so that $\frac{1}{2}<\frac{1}{2}+\tilde{a}_n<1\ \forall\ n$).  Now noting that$$\sum_{n=1}^{\infty}e^{-\tilde{\lambda}_n\big(\frac{1}{4\tilde{a}_n}-\frac{1}{2}\big)}\leq\sum_{n=1}^{\infty}e^{-n^{1/3}\big(\frac{3n^{1/3}}{4}-\frac{1}{2}\big)}+\sum_{n=1}^{\infty}e^{-n^{1/3}\big(\frac{1}{4a_n}-\frac{1}{2}\big)}+\sum_{n=1}^{\infty}e^{-\lambda_n\big(\frac{3n^{1/3}}{4}-\frac{1}{2}\big)}+\sum_{n=1}^{\infty}e^{-\lambda_n\big(\frac{1}{4a_n}-\frac{1}{2}\big)}<\infty$$(where the finiteness of the middle two sums on the right of the inequality follows from the fact that $a_n<\frac{1}{2}$ and $\lambda_n>0\ \forall\ n\geq 1$), it follows from the proof of \eqref{impfJRmodel} for the case where $\lambda_n$ and $a^{-1}_n$ are $O\big(n^{1/3}\big)$, that$$\sum_{n=2}^{\infty}e^{-\sum_{j=1}^{n-1}\lambda_{n-j}\big(1-\frac{4a_{n-j}}{1+2a_{n-j}}\big)^j}\leq\sum_{n=2}^{\infty}e^{-\sum_{j=1}^{n-1}\tilde{\lambda}_{n-j}\big(1-\frac{4\tilde{a}_{n-j}}{1+2\tilde{a}_{n-j}}\big)^j}<\infty$$(where the first inequality follows from the fact that $\tilde{\lambda_j}\leq\lambda_j$ and $\tilde{a}^{-1}_j\leq a^{-1}_j$).  Hence, this establishes \eqref{impfJRmodel} for the general case, and thus completes the proof of the theorem.
  \end{proof}

\section*{Acknowledgements}

The author would like to thank Toby Johnson for providing extensive background on the frog model.

\end{document}